\documentclass[
11pt]{article}
\usepackage{latexsym}
\usepackage{amsfonts}
\usepackage{amsmath}
\usepackage{amsthm}
\usepackage{amssymb}
\usepackage{graphicx}
\usepackage{color}
\usepackage{enumerate}

\usepackage{mathtools}

\usepackage{graphics}
\usepackage{epsfig}

\usepackage{calrsfs}
\usepackage{esint}

\usepackage
[
final,
color]{showkeys}%
\definecolor{labelkey}{cmyk}{1,1,1,.8}
\definecolor{refkey}{cmyk}{1,1,1,.8}
\definecolor{labelkey}{gray}{0.2} 
\definecolor{refkey}{gray}{0.2} 
\definecolor{labelkey}{rgb}{0,0,1}

\definecolor{refkey}{rgb}{0,0,1}
\definecolor{citekey}{gray}{.05}

\usepackage{hyperref}%

\newtheorem{thm}{Theorem}[section]
\newtheorem{lemma}[thm]{Lemma}

\newtheorem{proposition}[thm]{Proposition}
\newtheorem{definition}[thm]{Definition}
\newtheorem{corollary}[thm]{Corollary}
\newtheorem{example}[thm]{Example}
\newtheorem{exercise}[thm]{Exercise}
{\theoremstyle{definition}
\newtheorem{remark}[thm]{Remark}}

\newcommand{\bth}{\begin{theorem}}
\newcommand{\ethe}{\end{theorem}}

\newcommand{\bre}{\begin{remark}\em }
\newcommand{\ere}{\end{remark}}

\newcommand{\ble}{\begin{lemma}}
\newcommand{\ele}{\end{lemma}}

\newcommand{\bde}{\begin{definition}}
\newcommand{\ede}{\end{definition}}
\newcommand{\bco}{\begin{corollary}}
\newcommand{\eco}{\end{corollary}}

\newcommand{\bpr}{\begin{proposition}}
\newcommand{\epr}{\end{proposition}}

\newcommand{\bexer}{\begin{exercise}}
\newcommand{\eexer}{\end{exercise}}

\newcommand{\bexam}{\begin{example}\rm }
\newcommand{\eexam}{\end{example}}

\newcommand{\bali}{\begin{align}}
\newcommand{\eali}{\end{align}}

\renewcommand\d{{\mathrm d}}
\renewcommand\d{\operatorname{\mathrm d}\!}

\newcommand\C{{\mathbb{C}}}
\newcommand\R{{\mathbb{R}}}
\newcommand\N{{\mathbb{N}}}
\newcommand\Z{{\mathbb{Z}}}

\renewcommand\P{{\operatorname{\mathsf{P}}}}
\newcommand\E{{\operatorname{\mathsf{E}}}}

\newcommand\ii[1]{{\,\operatorname{\mathsf{I}}\{#1\}}}
\renewcommand\ii{{\,\operatorname{\mathsf{I}}}}

\newcommand{\sgn}{\operatorname{sgn}}

\newcommand{\fl}[1]{\lfloor#1\rfloor}

\newcommand{\ce}[1]{\lceil#1\rceil}

\newcommand\D{{\operatorname{D}}}

\renewcommand{\Re}{\operatorname{\mathfrak{Re}}}
\renewcommand{\Im}{\operatorname{\mathfrak{Im}}}

\newcommand{\pv}{\mathrm{p.v.}}

\newcommand\G{{\mathcal{G}}}

\newcommand{\tR}{{\tilde{R}}}

\newcommand{\al}{\alpha}
\newcommand{\ga}{\gamma}
\newcommand{\Ga}{\Gamma}
\newcommand{\de}{\delta}
\newcommand{\De}{\Delta}
\newcommand{\la}{\lambda}
\newcommand{\ka}{\kappa}
\newcommand{\si}{\sigma}

\newcommand{\vep}{\varepsilon}

\newcommand\ov{\overline}

\allowdisplaybreaks

\newcommand{\lhalfeq}[1]{\mbox{$#1=$\kern-3.5pt{\textcolor{white}{\rule[1pt]{4pt}{3.4pt}}} } }

\newcommand{\rhalfeq}[1]{\mbox{$#1=$\kern-7.9pt{\textcolor{white}{\rule[1pt]{4pt}{3.4pt}}} } }

\makeatletter
\renewcommand{\lhalfeq}[1]{\mbox{\sbox0{$#1\vcenter{}$}\raisebox{\dimexpr\height-2\ht0\relax}[2\dimexpr\height-\ht0\relax][0pt]{$\m@th=$}\kern-3.5pt{\textcolor{white}{\rule[-1pt]{4pt}{3.8pt}}} } }

\renewcommand{\rhalfeq}[1]{\mbox{\sbox0{$#1\vcenter{}$}\raisebox{\dimexpr\height-2\ht0\relax}[2\dimexpr\height-\ht0\relax][0pt]{$\m@th=$}\kern-7.9pt{\textcolor{white}{\rule[-1pt]{4pt}{3.8pt}}} } }
\makeatother

\newcommand{\gtheq}{\mathrel{\mathpalette\xgtheq\relax}}
\newcommand{\xgtheq}[2]{%
  \vcenter{\hbox{%
    \oalign{$#1>$\cr\lhalfeq{#1}\hidewidth\cr}%
  }}
}
\newcommand{\ltheq}{\mathrel{\mathpalette\xltheq\relax}}
\newcommand{\xltheq}[2]{%
  \vcenter{\hbox{%
    \oalign{$#1<$\cr
    \rhalfeq{#1}\cr}%
  }}\kern-3pt
}

\title{Positive-part moments via the characteristic functions, and more general expressions} 
\author{ Iosif Pinelis\footnote{Department of Mathematical Sciences, 
Michigan Technological University, 
1400 Townsend Drive, 
Houghton, Michigan 49931, 
U.S.A. E-mail: 
ipinelis@mtu.edu. }
}
\begin{document}
\maketitle

\begin{abstract}
A unifying and generalizing approach to representations of the positive-part and absolute moments ${\operatorname{\mathsf{E}}} X_+^p$ and ${\operatorname{\mathsf{E}}}|X|^p$ of a random variable $X$ for real $p$ in terms of the characteristic function (c.f.) of $X$, as well as to related representations of the c.f.\ of $X_+$, generalized moments ${\operatorname{\mathsf{E}}} X_+^p e^{iuX}$, truncated moments, and the distribution function is provided. Existing and new representations of these kinds are all shown to stem from a single basic representation. Computational aspects of these representations are addressed. 

{\it Keywords:} 
characteristic functions, positive-part moments, 
absolute moments, 
truncated moments, fractional derivatives. 
\vspace{2mm}\\
2010 Mathematics Subject Classification: Primary 60E10, Secondary 60E07; 62E15; 60E15; 91B30

%
\end{abstract}

\pagecolor{white}

\tableofcontents

\section{Introduction}
Fourier analysis is a powerful tool, used widely in mathematics, sciences, and engineering. In probability theory, the Fourier transform of (the distribution of) a random variable (r.v.) is called the characteristic function (c.f.), which has a number of well-known useful properties, including the following: (i) the c.f.\ of the sum of independent r.v.'s is the product of the c.f.'s of the summands; (ii) in view of the various inversion formulas, any distribution is completely determined by its c.f.; (iii) the moments of a distribution (when they exist) can be found by (possibly repeated) differentiation of the c.f. at $0$. 

These properties of the c.f.\ have been used extensively to derive various versions of the central limit theorem and its refinements; see e.g.\ \cite{esseen42,zolotarev67,prawitz75-2,tyurin11,shev11,more-nonunif}. The c.f.\ is a natural tool of choice in studies of L\'evy processes and infinitely divisible distributions, in particular stable laws, which are widely used in modeling in finance. The famous Spitzer identity \cite{spitzer}, which relates the distributions of the cumulative maxima of the partial sums of independent identically distributed r.v.'s with the distributions of the positive parts of those sums, is also naturally expressed in terms of c.f.'s. Expressions of the c.f.\ of the positive part $X_+:=0\vee X$ of a r.v.\ $X$ in terms of the c.f. of $X$ were recently given in \cite{c.f.-pos_publ}; in particular, that resulted in a more explicit form of the Spitzer identity. 

There have been a number of results that provide general expressions for absolute moments of a r.v.\ in terms of its c.f.; see e.g.\ \cite{zolot-mellin,bahr65,lukacs,kawata,pet75,wolfe:1975a,laue:1980}. Such results were used, in particular, to provide 
(i) bounds on the rate of convergence of the absolute moments of the standardized sums of independent r.v.'s to the corresponding moments of the standard normal distribution \cite{bahr-converg65}; 
(ii) upper bounds on the absolute moments of the sums of independent zero-mean r.v.'s \cite{bahr65,rosenthal_AOP}; 
(iii) explicit expressions of the absolute moments of stable laws \cite{zolot-mellin}; 
(iv) upper and lower bounds (differing from each other only by a constant factor) on the Lorentz norm of a r.v.\ \cite{braverman96}; 
(v) characterizations of a probability distribution in terms of certain related absolute moments \cite{braverman86,bms95}. 

In contrast with the case of the absolute moments, there appear to have been rather few papers providing expressions for the moments of the positive part $X_+$ of a r.v.\ $X$ in terms of the c.f.\ of $X$. In this regard, we can only mention articles \cite{zolot-mellin,brown:1970,pinelis:2011}. 

However, it should be clear that such results for $X_+$ are in many cases of 
greater value than those for $|X|$, since $|X|^p=X_+^p+(-X)_+^p$ for all $p>0$. 
In particular, positive-part moments are important in applications in finance -- see e.g.\ \cite{q-bounds-published,winzor} and many further references there. Indeed, $(X-K)_+$  and $(K-X)_+$ are, respectively, the current values of a call option and a put option given the current underlying stock price $X$ and the strike price $K$ of the option.
Also, the mentioned expressions in \cite{pinelis:2011} for the moments of the positive part $X_+$ of a r.v.\ $X$ in terms of the c.f.\ of $X$ were instrumental in the development in \cite{rosenthal_AOP} of  a calculus of variations of generalized moments of infinitely divisible distributions with respect to variations of the L\'evy characteristics, which in turn played a crucial role in obtaining exact Rosenthal-type bounds for sums of independent random variables.

In a large number of extremal problems of probability and statistics, it is the moments of the positive part of a r.v.\ (rather than those of the absolute value) that are of interest -- see e.g.\ \cite{eaton1,T2,pin98,bent-ap,
shaked-shanti,normal,asymm,pin-hoeff-published,q-bounds-published,left_publ,cones}. 
%
For instance, in \cite{q-bounds-published} a spectrum of coherent measures of financial risk was proposed and studied. These risk measures are upper bounds on the $(1-q)$-quantiles of a r.v.\ $X$ (with $q\in(0,1)$), which can be expressed by the following variational formula: 
\begin{equation}\label{eq:Q=inf}
Q_\al(X;q)=\inf_{t\in\R}\big[t+\big(\E(X-t)_+^\al/q\big)^{1/\al}\big],  
\end{equation} 
where $\al\in(0,\infty)$ is the risk sensitivity parameter. 
This spectrum of upper bounds on the quantiles of $X$ is naturally based on (and, in a certain sense, dual to) the previously studied spectrum of upper bounds on the tails of (the distribution of) $X$ given by the formula 
\begin{equation}\label{eq:P old}
P_\al(X;x)=\inf_{t\in\R}\frac{\E(X-t)_+^\al}{(x-t)_+^\al}; 
\end{equation}
see e.g.\ \cite{pin98,pin-hoeff-arxiv-reftoAIHP} and bibliography there. 
As was pointed out by 
Rockafellar and Uryasev \cite{rocka-ur02}, such variational representations are ``[m]ost important[] for applications''.  Indeed, such representations allow of a comparatively easy incorporation of the risk measures into more specialized optimization problems, with additional restrictions  
on the r.v.\ $X$; see \cite[Section~4.3]{q-bounds-published} for details. 

Clearly, the use of formulas such as \eqref{eq:Q=inf} and \eqref{eq:P old} requires fast and efficient calculations of the positive-part moments $\E(X-t)_+^\al$. In many cases, including the mentioned applications to risk assessment and management in finance, a good way to conduct such calculations is by using representations of the positive-part moments in terms of the c.f., such as the ones provided in the present paper; cf.\ also the discussion on the computational effectiveness in the introduction in \cite{pinelis:2011}. The important computational aspects of the representations obtained in the present paper will be addressed in detail in Section~\ref{comput}, which puts it in distinction with other papers in this area. 

Another distinctive feature of this paper is a unifying and generalizing method of obtaining representations of the positive-part moments moments in terms of the c.f., based on little more than a simple idea of homogeneity. This quickly gives us Theorem~\ref{th:homo}, from which all the other results presented in this paper follow, both the existing and new ones. In particular, the mentioned expression of the c.f.\ of $X_+$ in terms of the c.f.\ of $X$ given in \cite{c.f.-pos_publ} is shown to be a corollary to Theorem~\ref{th:homo}; the same is true concerning other known results, including Gurland's inversion formula \cite{gurland:1948}, expressing the distribution function in terms of the c.f. 

The only problem that Theorem~\ref{th:homo} leaves out is that of the evaluation of the constant factors $c_p^\pm(g)$ in \eqref{eq:X homo}. In a number of applications (e.g.\ in the characterization problems studied in \cite{braverman86,bms95,braverman93,braverman96}), these constant factors are of no importance, and some known representation formulas for absolute moments in terms of the c.f.\ (such as the ones in \cite[Theorem~11.4.4]{kawata} and \cite[Theorem~2]{wolfe:1973}) leave the constant factors unevaluated. 

In many cases, though, explicit values of the constants are important. We shall see that for all the results presented in the paper, the constant factors can be evaluated in a more or less straightforward manner based just on \cite[Corollary~2]{pinelis:2011} specialized to the case when the r.v.\ $X$ takes only one real value.  

Thus, we provide a completely unified approach to the problem of representing the positive-part moments (and hence the absolute moments) in terms of the c.f. In fact, we also provide similar representations for the moments of the more general form $\E X_+^p e^{iuX}$ for any real $p$ and $u$
, as well as for truncated moments. 

\section{General results}

Take any real $p$. Let $\G_p$ denote the class of all locally integrable Borel-measurable functions $g\colon\R\to\C$ such that 
$g(0)=0$ and there exist finite limits 
\begin{equation}\label{eq:c^pm}
	c^+_p(g):=\int_{0+}^{\infty-} \frac{g(t)}{t^{p+1}}\d t\quad\text{and}\quad 
	c^-_p(g):=\int_{0+}^{\infty-} \frac{g(-t)}{t^{p+1}}\d t.
\end{equation}
Here and in what follows, 
$\int_{0+}^{\infty-}:=\lim_{\vep\downarrow0,\,T\uparrow\infty}\int_\vep^T$.  
Similarly,  $\int_0^{\infty-}:=\lim_{T\uparrow\infty}\int_0^T$ and  $\int_{0+}^\infty:=\lim_{\vep\downarrow0}\int_\vep^\infty$. 
%
Whenever a claim about an integral of the form $\int_a^b$ for some $a$ and $b$ such that $0\le a<b\le\infty$ is made in this paper, it includes by default the statement 
that the integral exists in the Lebesgue sense (and is finite). 

\smallskip
Obviously, if a function $g\in\G_p$ is even, then $c^-_p(g)=c^+_p(g)$; and 
if $g\in\G_p$ is odd, then $c^-_p(g)=-c^+_p(g)$. 

Note also that any $g\in\G_p$ will necessarily satisfy the boundedness condition 
\begin{equation}\label{eq:bounded}
	S_p(g):=\sup\Big\{\Big|\int_\vep^T \frac{g(tx)}{t^{p+1}}\d t\Big|
\colon0<\vep<T<\infty\text{ and }x\in\{1,-1\}\Big\}<\infty. 
\end{equation}

Let also 
\begin{gather*}
	x_+^p:=
	\begin{cases}x^p&\text{ if }x>0,\\0&\text{otherwise,}\end{cases} \qquad
	x_-^p:=(-x)_+^p=\begin{cases}(-x)^p&\text{ if }x<0,\\0&\text{otherwise},\end{cases} \qquad
	x^{[p]}:=x_+^p-x_-^p.  
\end{gather*}
Here and in the sequel, $x$ stands for an arbitrary real number (unless specified otherwise). 
Note that 
\begin{equation}\label{eq:|x|^p}
	|x|^p=x_+^p+x_-^p\ \text{if $x\ne0$ or $p\ge0$;} 
\end{equation}
however, $x_+^p+x_-^p=|x|^p\ii\{x\ne0\}$ for all real $x$ and $p$, so that 
$|x|^p=\infty\ne0=x_+^p+x_-^p$ if $x=0$ and $p<0$. 
Here and in the sequel, we assume the 
conventions  
\begin{equation}\label{eq:conventions}
	\text{$0^0=0$,\quad $0^p=\infty$ if $p<0$,\quad $0\cdot\infty=\infty\cdot0=0$,\quad and $a\cdot\infty=\infty\cdot a=\infty$ for real $a>0$. }
\end{equation}

The following simple proposition is our starting point. 

\begin{proposition}\label{prop:homo}
For any $g\in\G_p$ 
\begin{equation}\label{eq:homo}
	\int_{0+}^{\infty-} \frac{g(tx)}{t^{p+1}}\d t=c^+_p(g)x_+^p+c^-_p(g)x_-^p.  
\end{equation}
\end{proposition}

\begin{proof}
If $x=0$, \eqref{eq:homo} follows immediately from the condition $g(0)=0$. 
If $x\ne0$, \eqref{eq:homo} follows immediately by the change of variables: introducing $u:=t|x|$, one has 
\begin{equation*}
	\int_{0+}^{\infty-} \frac{g(tx)}{t^{p+1}}\d t
	=|x|^p	\int_{0+}^{\infty-} \frac{g(u\sgn x)}{u^{p+1}}\d u
	=c^+_p(g)x_+^p+c^-_p(g)x_-^p, 
\end{equation*}
by \eqref{eq:c^pm}. 
\end{proof}

\begin{remark}\label{rem:flip}
It follows from \eqref{eq:homo} that 
$\int_{0+}^{\infty-} \frac{g(-tx)}{t^{p+1}}\d t=c^+_p(g)x_-^p+c^-_p(g)x_+^p$. Solving the latter equation together with \eqref{eq:homo} for $x_+^p$ and $x_-^p$, one obtains integral expressions for $x_\pm^p$ and hence for $|x|^p$ and $x^{[p]}$:
\begin{align}
	x_+^p&=\frac1{c^+_p(g)^2-c^-_p(g)^2}
	\int_{0+}^{\infty-} \frac{c^+_p(g)g(tx)-c^-_p(g)g(-tx)}{t^{p+1}}\d t, \notag \\
	x_-^p&=\frac1{c^+_p(g)^2-c^-_p(g)^2}
	\int_{0+}^{\infty-} \frac{c^+_p(g)g(-tx)-c^-_p(g)g(tx)}{t^{p+1}}\d t, \notag \\ 
	|x|^p&=\frac1{c^+_p(g)+c^-_p(g)}
	\int_{0+}^{\infty-} \frac{g(tx)+g(-tx)}{t^{p+1}}\d t\quad\text{if $x\ne0$ or $p\ge0$}, \notag \\ 
	x^{[p]}&=\frac1{c^+_p(g)-c^-_p(g)}
	\int_{0+}^{\infty-} \frac{g(tx)-g(-tx)}{t^{p+1}}\d t, \notag 
\end{align}
provided that $c^+_p(g)^2\ne c^-_p(g)^2$. 
Similarly and more generally, 
\begin{equation}\label{eq:x_+^p=}
	x_+^p=\frac1{c^+_p(g_1)c^+_p(g_2)-c^-_p(g_1)c^-_p(g_2)}
	\int_{0+}^{\infty-} \frac{c^+_p(g_2)g_1(tx)-c^-_p(g_1)g_2(tx)}{t^{p+1}}\d t  
\end{equation}
for any $g_1$ and $g_2$ in $\G_p$ such that 
\begin{equation}\label{eq:ne0}
	c^+_p(g_1)c^+_p(g_2)\ne c^-_p(g_1)c^-_p(g_2). 
\end{equation}
Similar integral expressions for $x_-^p$, $|x|^p$, and $x^{[p]}$ can be easily obtained either directly or from \eqref{eq:x_+^p=}. 

In particular, if $g_1\in\G_p$ is odd and $g_2\in\G_p$ is even, then identity \eqref{eq:x_+^p=} can be rewritten as 
\begin{equation*}
	x_+^p=\frac1{2c^+_p(g_1)c^+_p(g_2)}
	\int_{0+}^{\infty-} \frac{c^+_p(g_1)g_2(tx)+c^+_p(g_2)g_1(tx)}{t^{p+1}}\d t,  
\end{equation*}
provided that 
\begin{equation}\label{eq:ne0,even-odd}
	c^+_p(g_1)c^+_p(g_2)\ne 0, 
\end{equation}
in which case 
\begin{align}
	x_-^p&=\frac1{2c^+_p(g_1)c^+_p(g_2)}
	\int_{0+}^{\infty-} \frac{c^+_p(g_1)g_2(tx)-c^+_p(g_2)g_1(tx)}{t^{p+1}}\d t, \notag \\ 
	|x|^p&=\frac1{c^+_p(g_2)}
	\int_{0+}^{\infty-} \frac{g_2(tx)}{t^{p+1}}\d t\quad\text{if $x\ne0$ or $p\ge0$}, \label{eq:|x|^p=int} \\ 
		x^{[p]}&=\frac1{c^+_p(g_1)}
	\int_{0+}^{\infty-} \frac{g_1(tx)}{t^{p+1}}\d t. \notag 
\end{align}
\end{remark}



Let $X$ be a real-valued r.v., and let $Y$ be a complex-valued r.v.  

\begin{thm}\label{th:homo} 
Suppose that 
\begin{equation}\label{eq:cond1}
	\E|Y|\,|X|^p<\infty 
\end{equation}
and a function  
$g\in\G_p$ satisfies the condition 
\begin{equation}\label{eq:cond2}
	\text{$\E|Y|\int_\vep^T |g(tX)|\d t<\infty$ for all real $\vep$ and $T$ such that $0<\vep<T$. } 
\end{equation}
Then
\begin{equation}\label{eq:X homo}
	\int_{0+}^{\infty-} \frac{\E Yg(tX)}{t^{p+1}}\d t
	=c^+_p(g)\E YX_+^p+c^-_p(g)\E YX_-^p. 
\end{equation}
\end{thm}

\begin{proof}
By the change of 
variables, $u:=t|x|$, and in view of \eqref{eq:bounded}, 
\begin{equation}\label{eq:<S|x|^p}
	\Big|\int_\vep^T \frac{g(tx)}{t^{p+1}}\d t\Big|
	=|x|^p\,	\Big|\int_{|x|\vep}^{|x|T} \frac{g(u\sgn x)}{u^{p+1}}\d u\Big|
	\le S_p(g)|x|^p 
\end{equation}
for all real $\vep$ and $T$ such that $0<\vep<T$ and 
all real $x\ne0$. 
Since $g(0)=0$ and $S_p(g)\ge0$, the inequality $\Big|\int_\vep^T \frac{g(tx)}{t^{p+1}}\d t\Big|\le S_p(g)|x|^p$ holds for $x=0$ as well. 

Using now \eqref{eq:cond2}, the Fubini theorem, \eqref{eq:<S|x|^p} with the condition $S_p(g)<\infty$ in \eqref{eq:bounded}, \eqref{eq:cond1}, the dominated convergence theorem, and Proposition~\ref{prop:homo}, one concludes the proof as follows: for all real $\vep$ and $T$ such that $0<\vep<T$, 
\begin{equation}\label{eq:lim}
\begin{multlined}
	\int_\vep^T \frac{\E Yg(tX)}{t^{p+1}}\d t
=	\E Y\int_\vep^T \frac{g(tX)}{t^{p+1}}\d t  
\underset{\vep\downarrow0,\,T\uparrow\infty}\longrightarrow
\E Y\int_{0+}^{\infty-} \frac{g(tX)}{t^{p+1}}\d t \\ 
\rule{0pt}{13pt}=\E Y\big(c^+_p(g)X_+^p+c^-_p(g)X_-^p\big)
=c^+_p(g)\E YX_+^p+c^-_p(g)\E YX_-^p. 
\end{multlined} 	
\end{equation}
\end{proof}

Note that \eqref{eq:homo} is a special case of \eqref{eq:X homo}, with $Y=1$ and $X=x$. 

Similarly to Remark~\ref{rem:flip}, one obtains the following corollary of Theorem~\ref{th:homo}. 

\begin{corollary}\label{cor:E X_+^p} 
Suppose that the r.v.'s $X$ and $Y$ and functions $g_1$ and $g_2$ are such that the conditions of Theorem~\ref{th:homo} are satisfied when the function $g$ there is replaced by either $g_1$ and $g_2$. Then 
\begin{equation*}
	\E YX_+^p=\frac1{c^+_p(g_1)c^+_p(g_2)-c^-_p(g_1)c^-_p(g_2)}
	\int_{0+}^{\infty-} \frac{c^+_p(g_2)\E Yg_1(tX)-c^-_p(g_1)\E Yg_2(tX)}{t^{p+1}}\d t,  
\end{equation*}
provided that \eqref{eq:ne0} holds. 
The similar expressions for $\E YX_-^p$, $\E Y|X|^p$, and $\E YX^{[p]}$ also hold. 

In particular, when the functions $g_1$ and $g_2$ are odd and even, respectively, and the condition \eqref{eq:ne0,even-odd} holds, then 
\begin{align}
	\E YX_+^p&=\frac1{2c^+_p(g_1)c^+_p(g_2)}
	\int_{0+}^{\infty-} \frac{c^+_p(g_1)\E Yg_2(tX)+c^+_p(g_2)\E Yg_1(tX)}{t^{p+1}}\d t, \label{eq:E YX_+^p} \\ 
	\E YX_-^p&=\frac1{2c^+_p(g_1)c^+_p(g_2)}
	\int_{0+}^{\infty-} \frac{c^+_p(g_1)\E Yg_2(tX)-c^+_p(g_2)\E Yg_1(tX)}{t^{p+1}}\d t, \label{eq:E YX_-^p} \\ 
	\E Y|X|^p&=\frac1{c^+_p(g_2)}
	\int_{0+}^{\infty-} \frac{\E Yg_2(tX)}{t^{p+1}}\d t, \label{eq:E Y|X|^p} \\ 
		\E YX^{[p]}&=\frac1{c^+_p(g_1)}
	\int_{0+}^{\infty-} \frac{\E Yg_1(tX)}{t^{p+1}}\d t. \label{eq:E YX^{[p]}}
\end{align}
\end{corollary}

\begin{proof}
Here only identity \eqref{eq:E Y|X|^p} may need justification -- cf.\ \eqref{eq:|x|^p=int}. Note that the right-hand side of \eqref{eq:E Y|X|^p} is the sum of those of \eqref{eq:E YX_+^p} and \eqref{eq:E YX_-^p}, whereas the difference between the left-hand side of \eqref{eq:E Y|X|^p} and the sum of those of \eqref{eq:E YX_+^p} and \eqref{eq:E YX_-^p} is $\E Y|X|^p\ii\{X=0\}$. 

It remains to notice that this difference is $0$, in view of \eqref{eq:cond1} and \eqref{eq:conventions}. Indeed, this is clearly so when $p\ge0$. If $p<0$, then $\infty\cdot\E|Y|\ii\{X=0\}=\E|Y||X|^p\ii\{X=0\}\le\E|Y||X|^p<\infty$, whence $\E|Y|\ii\{X=0\}=0$, and so, $\E|Y||X|^p\ii\{X=0\}=0$. 
%
\end{proof}

Taking specific functions 
$g$, $g_1$, $g_2$ in Theorem~\ref{th:homo} and Corollary~\ref{cor:E X_+^p}, one obtains an infinite variety of expressions for the generalized moments $\E YX_+^p$, $\E YX_-^p$, $\E Y|X|^p$, and $\E YX^{[p]}$ in terms of the generalized moments of the form $\E Yg(tX)$; as will be discussed later in this paper, this variety includes both known and new results. 

Of course, for such expressions to be useful, the moments of the form $\E Yg(tX)$ should be significantly more easily treatable than the positive-part moments $\E YX_+^p$ and the like. 
For instance, if $Y=1$ and $g(t)$ is a polynomial in $t$, $e^{ivt}$, and
$e^{-ivt}$ for some nonzero real $v$, then $\E Yg(tX)$ is a linear
combination of the values of the derivatives of various orders of the
c.f.\ of $X$ at points of the lattice $v\Z$ of
all integer multiples of $v$. 
In what follows, 
in most cases (with the exception of Corollary~\ref{cor:odd n,0r0}) we shall assume that 
\begin{equation*}
	Y=h(X), 
\end{equation*}
for a Borel-measurable function $h\colon\R\to\C$.  


\begin{thm}\label{th:lin comb}
Let $X$ be a r.v.\ and let $f$ denote the c.f.\ of $X$. 
Also, let $J$ be a finite set. 
Suppose that 
\begin{enumerate}[$($i$)$]
	\item\label{h=} $h(x)=\ov{x^r}\,e^{iux}$ for some real $r$ and $u$ and all real $x$;
	\item\label{g=} $g(t)=\sum\limits_{j\in J} a_j\,\ov{t^{q_j}}\,e^{iv_jt}$, for some functions $J\ni j\mapsto a_j\in\C\setminus\{0\}$, $J\ni j\mapsto q_j\in[0,\infty)$, and $J\ni j\mapsto v_j\in\R$, and for all real $t$;  
	\item\label{p>} for each $j\in J$, one has $p>q_j-1$ if $v_j\ne0$, and $p>q_j$ if $v_j=0$;
	\item\label{O} $|g(t)|=O(|t|^q)$ for some real $q>0\vee p$ and over all $t$ in a neighborhood of $0$; 
	\item\label{E<infty} $\E|X|^{r+p} + \E|X|^{r+q_j}<\infty$ for all $j\in J$. 
\end{enumerate}
Then 
\begin{equation}\label{eq:lin comb}
	\frac1{i^r}\int_0^{\infty-} \sum_{j\in J} a_j\, (-it)^{q_j}\, f^{(r+q_j)}(u+v_j t)
	\frac{\d t}{t^{p+1}}
	=c^+_p(g)\,\E X_+^{r+p} e^{iuX}+e^{-i\pi r}\,c^-_p(g)\,\E X_-^{r+p} e^{iuX}. 
\end{equation}
The limit $\int_0^{\infty-}$ in the left-hand side of \eqref{eq:lin comb} can be replaced by the Lebesgue integral $\int_0^{\infty}$ in the case when the condition \eqref{p>} above is satisfied in the strong sense that 
$p>q_j$ for all $j\in J$. 
\end{thm}

In part \eqref{h=} of Theorem~\ref{th:lin comb} and in the sequel, the
horizontal bar over an expression denotes, as usual, its complex
conjugate. 
In \eqref{eq:lin comb}, we use the notation $f^{(p)}$ for the fractional derivative of an arbitrary order $p\in\R$ of the
function $f$. If $f$ is the c.f.\ of a r.v.\ $X$ with $\E|X|^p<\infty$, then 
\begin{equation}\label{eq:f^{(q)}=}
f^{(p)}(t) = i^p\,\E\ov{X^p} e^{itX} 	
\end{equation}
for all real $p$ and $t$. 
This identity was
presented in \cite{wolfe:1975a}. 
See Appendix~\ref
{sec:appd:B} for details. 
In this paper, we follow the usual convention that the (principal value of the) argument $\arg z$ of a nonzero complex number $z$ belongs to the interval $(-\pi,\pi]$, and then $z^p$ is understood as $|z|^p e^{ip\arg z}$. In \cite{wolfe:1975a}, the argument $\arg z$ of a nonzero complex number $z$ is assumed to belong to the interval $[0,2\pi)$. 
However, both definitions result in the same value of $X^p$ in \eqref{eq:f^{(q)}=} -- because, according to either definition, $\arg x=0$ if $x>0$ and $\arg x=\pi$ if $x<0$. 
The suggestion to consider negative values of $p$ as well and the reference to \cite{wolfe:1975a} is due to M.\ Matsui \cite{muneya}. 

Note that, by the condition \eqref{g=} of Theorem~\ref{th:lin comb}, $|g(t)|=O(|t|^{q_{\min}})$ over all real $t$, where $q_{\min}:=\min_{j\in J}q_j$. In the 
applications of Theorem~\ref{th:lin comb} that will be presented in this paper, $q_{\min}$ will usually be $0$ or $1$. However, because of mutual near-cancellation of the summands $a_j\,t^{q_j}\,e^{iv_jt}$ for $t$ near $0$, the value of $q$ in the condition \eqref{O} of Theorem~\ref{th:lin comb} may be much greater than $1$ in 
some of those applications. 
In turn, this allows the value of $p$ to be rather large and still satisfy the restriction $q>0\vee p$ in the condition \eqref{O}; in particular, this will be so for the function $g$ as in \eqref{eq:g_n}
, which will result in faster convergence of the corresponging integral near $\infty$. 

\begin{proof}[Proof of Theorem~\ref{th:lin comb}]
By the conditions \eqref{g=} and \eqref{O} in Theorem~\ref{th:lin comb}, $g$ is a locally integrable Borel-measurable function with $g(0)=0$. 
The condition \eqref{O} also implies that $\int_0^1 \big|\frac{g(\pm t)}{t^{p+1}}\big|\d t<\infty$. 
Using (say) integration by parts, one can easily see that there exist finite limits 
$\int_1^{\infty-} \frac{t^{q_j}\,e^{\pm iv_jt}}{t^{p+1}}\d t$ whenever $p>q_j-1$ and $v_j\ne0$. 
Also, $\int_1^{\infty} \big|\frac{t^{q_j}\,e^{\pm iv_jt}}{t^{p+1}}\big|\d t<\infty$ whenever $p>q_j$. 
So, in view of the conditions \eqref{g=} and \eqref{p>}, there exist finite limits $c^\pm_p(g)$, as defined in \eqref{eq:c^pm}. Thus, $g\in\G_p$. 

The function $h\colon\R\to\C$ as defined by the condition \eqref{h=} is clearly Borel-measurable. 
Also, 
by the condition \eqref{E<infty} in Theorem~\ref{th:lin comb}, 
the condition \eqref{eq:cond1} holds with $Y=h(X)$. 
The condition \eqref{eq:cond2} also holds, in view of the conditions \eqref{h=}, \eqref{g=}, and \eqref{E<infty}. 

Thus, by Theorem~\ref{th:homo}, one has \eqref{eq:X homo}. 
It remain
s to observe that, in view of \eqref{eq:f^{(q)}=} and the conditions \eqref{h=} and \eqref{g=}, the identity \eqref{eq:X homo} can be easily rewritten as \eqref{eq:lin comb}; at that, one uses the identity $\ov{x^r}=e^{-i\pi r}x_-^r$ for the case $x<0$. 
\end{proof}


As illustrated in the next section, 
Theorem~\ref{th:lin comb} is an omnibus result: 
specifying the parameters involved in the statement of Theorem~\ref{th:lin comb} -- $p$, $r$, $u$, $J$, $(a_j)$, $(q_j)$, and $(v_j)$ -- one can obtain known results and their extensions, as well as a large variety of substantially new results. At that, it may be desirable to have ``closed form'' expressions for the constants $c^\pm_p(g)$, which appear on the right-hand side of \eqref{eq:lin comb}.  


\section{Special cases of \texorpdfstring{Theorem~\ref{th:lin comb}}{th:lin comb}}

\subsection{Characteristic function of \texorpdfstring{$X_+$}{c.f. of pos part}} In Theorem~\ref{th:lin comb}, let $r=p=0$ and $g(t)=\sin t=\frac1{2i}(e^{it}-e^{-it})$ for real $t$. Let $f$ be the c.f.\ of $X$. Then Theorem~\ref{th:lin comb} yields 
\begin{equation*}
	\E e^{iuX}\sgn X=\frac1{\pi i}\int_0^{\infty-}[f(u+t)-f(u-t)]\,\frac{\d t}t. 
\end{equation*}
Using now the identity $2e^{iuX_+}=1+e^{iuX}+e^{iuX}\sgn X-\sgn X$ (presented in \cite[formula ~(8)]{c.f.-pos_publ}), one reproduces the following expression of the c.f.\ of $X_+$ in terms of the c.f.\ $f$ of $X$, obtained in \cite{c.f.-pos_publ}: 
\begin{equation*}
	\E e^{iuX_+}=\frac{1+f(u)}2+\frac1{2\pi i}\,\int_0^{\infty-}[f(u+t)-f(u-t)-f(t)+f(-t)]\,\frac{\d t}t
\end{equation*}
for real $u$. 

\subsection{Extensions of a result by Pinelis
}
Consider the identity \cite[Corollary~2]{pinelis:2011}
\begin{equation}\label{eq:0}
 \E X_+^p =
 \frac{\E X^k}2\ii\{p\in\N\}+
 \frac{\Ga(p+1)}\pi 
\,\int_0^\infty\Re
\frac{\E e_\ell(itX)}{(it)^{p+1}}\,\d t 
\end{equation} 
for any $p\in(0,\infty)$ and any r.v.\ $X$ with $\E|X|^p<\infty$, where 
\begin{equation}\label{eq:ell}
k:=k(p):=\fl p,\quad \ell:=\ell(p):=\ce{p-1},\quad\text{and}\quad 
e_m(z):=e^z-\sum
	_{j=0}^m\frac{z^j}{j!}.   
\end{equation}

Using this identity with the ``nonrandom'' r.v.'s $1$ and $-1$ in place of the r.v.\ $X$
, one obtains the constants 
\begin{equation}
\label{eq:c_pm,pin11}
	c^+_p(g)=1-\tfrac12\ii\{p\in\N\}\quad\text{and}\quad
	c^-_p(g)=\tfrac12\,(-1)^{k
	+1}\ii\{p\in\N\}
\end{equation}
for the function $g$ given by the formula 
{\definecolor{labelkey}{gray}{.7}
\definecolor{refkey}{gray}{.99}	
\begin{equation}\label{eq:g pin_pos}
g(t):=g_{\eqref{eq:g pin_pos}}(t):= 
	\frac{\Ga(p+1)}\pi
	\Re\frac{(R_\ell\exp)(0;it)}{i^{p+1}}
	=\frac{\Ga(p+1)}{2\pi}
	\Big(\frac{(R_\ell\exp)(0;it)}{i^{p+1}}+\frac{(R_\ell\exp)(0;-it)}{(-i)^{p+1}}\Big), 
\end{equation}
}
where 
\begin{equation}\label{eq:R_m}
	(R_m\psi)(u;\de):=\psi(u+\de)-\sum_{j=0}^m\psi^{(j)}(u)\,\frac{\de^j}{j!}
	=\frac{\de^{m+1}}{m!}\int_0^1(1-\al)^m f^{(m+1)}(u+\al\de) \d \al  
\end{equation}
is the $m$th-order Taylor remainder for a function $\psi$, 
so that 
\begin{equation}\label{eq:e_ell=R}
	(R_\ell\exp)(0;it)=e_\ell(it). 
\end{equation}

Note that the function $g=g_{\eqref{eq:g pin_pos}}$ satisfies the condition \eqref{O} of Theorem~\ref{th:lin comb} with $q=\ell+1$ for $p\in(0,\infty)\setminus\N$ and $q=\ell+2$ for $p\in\N$, and this function obviously satisfies the conditions \eqref{g=} and \eqref{p>} of Theorem~\ref{th:lin comb} as well, for appropriately defined parameters $J$, $(a_j)$, $(q_j)$, and $(v_j)$. 
Moreover, 
here all the $q_j$'s are nonnegative integers, with $\max_{j\in J}q_j=\ell<p$, 
so that the condition \eqref{p>} above is satisfied in the strong sense mentioned in the last sentence in the statement of Theorem~\ref{th:lin comb}.  
Thus, transcribing the terms of the form $\ov{t^{q_j}}\,e^{iv_jt}$ into the corresponding terms $(-it)^{q_j}\, f^{(r+q_j)}(u+v_j t)$ in the left-hand side of \eqref{eq:lin comb}, 
one immediately obtains the following extension of \eqref{eq:0}: 

\begin{corollary}\label{cor:pin_pos} 
Take any $p\in(0,\infty)$, $r\in\R$, and $u\in\R$. Take any r.v.\ $X$ such that $\E|X|^r+\E|X|^{r+p}<\infty$. 
Let $f$ be the c.f.\ of $X$. 
Then 
\begin{equation}\label{eq:pin_pos}
\E X_+^{r+p}e^{iuX}=\frac{f^{(p+r)}(u)}{2i^{p+r}} \ii\{p\in\N\}
+\frac{\Gamma (p+1)}{2 \pi i^r} \int_0^{\infty }
\Big(\frac{(R_\ell f^{(r)})(u;t)}{(it)^{p+1}}+\frac{(R_\ell f^{(r)})(u;-t)}{(-it)^{p+1}}\Big) \d t.   
\end{equation}
\end{corollary} 

From the reasoning preceding Corollary~\ref{cor:pin_pos}, it is clear that the integral in \eqref{eq:pin_pos} exists in the Lebesgue sense, which is (in general) in contrast with the integrals in such identities as \eqref{eq:X homo}. 

Under the condition of Corollary~\ref{cor:pin_pos}, one can rewrite \eqref{eq:pin_pos} in the ostensibly simpler form 
\begin{equation}\label{eq:pin_pos,2}
\E X_+^{r+p}e^{iuX}=\frac{f^{(k+r)}(u)}{2i^{k+r}} \ii\{p\in
\N\}
+\frac{\Gamma (p+1)}{2 \pi i^r} \fint_{-\infty}^\infty
\frac{(R_\ell f^{(r)})(u;t)}{(it)^{p+1}} \d t,  
\end{equation}  
where 
$\fint_{-\infty}^\infty:=\lim_{\vep\downarrow0}\big(\int_{-\infty}^{-\vep}+\int_\vep^\infty\big)$. 	
Actually, this integral exists in the Lebesgue sense if $p\in(0,\infty)\setminus\N$. However, even then the calculation of each of the ``pre-limit'' integrals $\int_{-\infty}^{-\vep}$ and $\int_\vep^\infty$ will be less stable than the calculation of the corresponding ``pre-limit'' integral for the integral in \eqref{eq:pin_pos}; of course, the latter ``pre-limit'' integral equals the sum of ``pre-limit'' integrals $\int_{-\infty}^{-\vep}$ and $\int_\vep^\infty$ for the integral in \eqref{eq:pin_pos,2}.

Taking $r=0$ and $u=0$ in \eqref{eq:pin_pos}, one gets back \eqref{eq:0} as a special case. 

Note that, for any given real value of the exponent $r+p$ in the left-hand side of \eqref{eq:pin_pos}, one has a great deal of flexibility in choosing $p\in(0,\infty)$ and $r\in\R$. 
Note also that 
the exponent $r+p$ may be any negative real number, if one chooses $r$ to be less than $-p$.


Another special case of Corollary~\ref{cor:pin_pos} is obtained by taking there any $p\in(0,1]$ and any $r\in\{0,1,\dots\}$, as is done in the proof of the following corollary.  

\begin{corollary}\label{cor:c.f.X_+^p}
Take any $p\in(0,\infty)$ and $u\in\R$. Take any r.v.\ $X$ such that $\E|X|^p<\infty$. 
Let $f$ be the c.f.\ of $X$. 
Then, with $\la:=\la(p):=p-\ell(p)=p-\ell$ \big(and 
$\ell$ as in \eqref{eq:ell}\big),  
\begin{equation}\label{eq:c.f.X_+^p}
	 \E X^p_+ e^{i u X} 
 = \frac{f^{(p)}(u)}{2i^p} \ii{\{p\in \N\}} +
 \frac{\Gamma(\la+1)}{2\pi i^\ell}
 \int_0^\infty \Big(\frac{f^{(\ell)}(u+t)-f^{(\ell)}(u)}{(it)^{\la+1}}
 +\frac{f^{(\ell)}(u-t)-f^{(\ell)}(u)}{(-it)^{\la+1}}
 \Big)\d t. 
\end{equation}
\end{corollary} 

\begin{proof} 
In Corollary~\ref{cor:pin_pos}, 
replace $p$ by an arbitrary $\ga\in(0,1]$ and also take any $r\in\{0,1,\dots\}$. 
By \eqref{eq:ell}, $\ell(\ga)=0$. So, \eqref{eq:pin_pos} yields 
\begin{multline}\label{eq:c.f.X_+^p,alt}
	 \E X^{r+\ga}_+ e^{i u X} 
 = \frac{f^{(r+\ga)}(u)}{2i^{r+\ga}} \ii{\{\ga+r\in \mathbb N\}} 
 \\
 +
 \frac{\Gamma(\ga+1)}{2\pi i^r}
 \int_0^\infty \Big(\frac{f^{(r)}(u+t)-f^{(r)}(u)}{(it)^{\ga+1}} 
 +\frac{f^{(r)}(u-t)-f^{(r)}(u)}{(-it)^{\ga+1}}
 \Big)\d t. 
\end{multline}
It remains to replace, in \eqref{eq:c.f.X_+^p,alt}, $r+\ga$ by $p$ and notice that 
$\ell(r+\ga)=r$, $\la(r+\ga)=\ga$, and $\big\{r+\ga\colon r\in\{0,1,\dots\},\;\ga\in(0,1]\big\}=(0,\infty)$. 
\end{proof}

\renewcommand{\bar}{\overline}

Versions of \eqref{eq:pin_pos} and \eqref{eq:c.f.X_+^p} with $u=0$ were obtained in \cite{muneya}. 

For $u=0$, formula \eqref{eq:c.f.X_+^p} can be simplified, using the identities $f^{(\ell)}(-t)=(-1)^\ell\bar{f}^{(\ell)}(t)$, $X_-^p=(-X)_+^p$, $|X|^p=X_+^p+X_-^p$, and $X^{[p]}=X_+^p-X_-^p$.  
So, one immediately obtains 
\begin{corollary}\label{cor:c.f.X_+^p,u=0}
In the conditions of Corollary~\ref{cor:c.f.X_+^p},  
\begin{align}
	 \E X^p_+ &= \frac{\E X^p}{2} \ii{\{p\in \N\}} +
 \frac{\Ga(\la+1)}{\pi} \int_0^\infty \Re\frac{f^{(\ell)}(t)-f^{(\ell)}(0)}{i^{p+1}}\,
 \frac{\d t}{t^{\la+1}}, \label{eq:c.f.X_+^p,u=0} \\ 
	 \E X^p_- &= \frac{(-1)^p\,\E X^p}{2} \ii{\{p\in \N\}} +
 \frac{\Ga(\la+1)}{\pi} \int_0^\infty \Re\frac{\bar{f}^{(\ell)}(t)-\bar{f}^{(\ell)}(0)}{i^{p+1}}\,
 \frac{\d t}{t^{\la+1}}, \label{eq:c.f.X_-^p,u=0} \\ 
 	 \E|X|^p &= \E X^p \ii{\Big\{\frac p2\in \N\Big\}} +
 2\frac{\Ga(\la+1)}{\pi}\,\Re(i^{p+1}) \int_0^\infty \frac{\Re\big(f^{(\ell)}(t)-f^{(\ell)}(0)\big)}{t^{\la+1}}\d t, \label{eq:c.f.|X|^p,u=0} \\ 
 	 \E X^{[p]} &= \E X^p \ii{\Big\{\frac{p+1}2\in \N\Big\}} +
 2\frac{\Ga(\la+1)}{\pi}\,\Re(i^p) \int_0^\infty \frac{\Im\big(f^{(\ell)}(t)-f^{(\ell)}(0)\big)}{t^{\la+1}}\d t. \label{eq:c.f.X^[p],u=0}  
\end{align}
\end{corollary} 

In the ``fractional'' case, when $p$ is a positive non-integer, one can rewrite the expression of $\E X^p_+ e^{i u X}$ in \eqref{eq:c.f.X_+^p} in terms of fractional derivatives of $f$: 

\begin{corollary}\label{cor:cor c.f.X_+^p}
Assume the conditions of Corollary~\ref{cor:c.f.X_+^p}, with any $p\in(0,\infty)\setminus\N$.  
Then for any $u\in\R$
\begin{align} 
\E X_+^p e^{i u X}&= \frac{1}{2i^{\ell+1} \sin
 \pi\la} \Big\{i^{\la}f^{(p)}(u)-(-1)^\ell\,
 i^{-\la}\bar{f}^{(p)}(-u)\Big\}, \label{eq:posi-p,u} \\ 
\E X_-^p e^{i u X} &=
\frac{1}{2i^{\ell+1} \sin \pi \la} \Big\{ i^{\la}\bar{f}^{(p)}(-u) -(-1)^\ell\,
 i^{-\la} f^{(p)}(u) \Big\}, \label{eq:nega-p,u} \\ 
\E|X|^p e^{i u X} &=
\frac{i^\la-(-1)^\ell\, i^{-\la}}{2i^{\ell+1} \sin \pi \la} 
\big\{f^{(p)}(u)+\bar{f}^{(p)}(-u)\big\}, \label{eq:abs-p,u} \\ 
\E X^{[p]} e^{i u X} &=
\frac{i^\la+(-1)^\ell\, i^{-\la}}{2i^{\ell+1} \sin \pi \la} 
\big\{f^{(p)}(u)-\bar{f}^{(p)}(-u)\big\}. \label{eq:sign-p,u} 
\end{align}
\end{corollary} 

\begin{proof} By \eqref{eq:c.f.X_+^p}, the condition $p\in(0,\infty)\setminus\N$, and \eqref{eq:D^p}, 
\begin{align*}
	 \E X^p_+ e^{i u X} 
& = 
 \frac{\Ga(\la+1)}{2\pi i^\ell}
 \int_0^\infty \Big(\frac{f^{(\ell)}(u+t)-f^{(\ell)}(u)}{(it)^{\la+1}}
 +\frac{f^{(\ell)}(u-t)-f^{(\ell)}(u)}{(-it)^{\la+1}}
 \Big)\d t \\
 &= \frac{\Ga(\la+1)}{2\pi i^\ell} \Big\{
 (-1)^{\ell} \int_0^\infty \frac{\bar f^{(\ell)}(-u-t)-\bar f^{(\ell)}(-u)}
 {(it)^{\la+1}}\d t 
 + \int_0^\infty
 \frac{f^{(\ell)}(u-t)-f^{(\ell)}(u)}{(-it)^{\la+1}}\d t  
\Big\} \\
 &= \frac{\Ga(\la+1)\Ga(-\la)}{2\pi i^\ell}\big\{
 (-1)^{\ell+1}\, i^{1-\la}\bar{f}^{(p)}(-u)+i^{1+\la}f^{(p)}(u)
 \big\} \\ 
 & = \frac{1}{2i^\ell \sin \pi\la}\big\{
 i^{\la-1}f^{(p)}(u)-(-1)^\ell\, i^{-\la-1}\bar{f}^{(p)}(-u)
 \big\},
\end{align*}
so that \eqref{eq:posi-p,u} is established. 
Now \eqref{eq:nega-p,u} follows by the identity $\E X^p_- e^{i u X}=\E(-X)^p_+ e^{i(-u)(-X)}$ for all $u\in\R$. 
Finally, \eqref{eq:abs-p,u} and \eqref{eq:sign-p,u} follow immediately from \eqref{eq:posi-p,u} and \eqref{eq:nega-p,u}. 
\end{proof}

The expressions in Corollary~\ref{cor:cor c.f.X_+^p} can be simplified in the case $u=0$. Alternatively, one can obtain the same result by taking the special case of $p\notin\N$ in Corollary~\ref{cor:c.f.X_+^p,u=0}: 
\begin{corollary}\label{th:short}
Let $ p \in (0,\infty)
\setminus \mathbb{N}$ and 
assume that $\E|X|^p<\infty$. Then 
\begin{align}
	\E X_+^p&=\frac{(-1)^{\ell+1}}{\sin\pi\la}\,\Re \big(i^{p+1}f^{(p)}(0)\big), \label{eq:posi-p} \\ 
	\E X_-^p&=\frac{(-1)^{\ell+1}}{\sin\pi\la}\,\Re \big(i^{p+1}\bar{f}^{(p)}(0)\big), \label{eq:nega-p} \\ 
	\E|X|^p&=2\frac{(-1)^{\ell+1}}{\sin\pi\la}\,\Re(i^{p+1})\,\Re f^{(p)}(0),  \label{eq:abs-p} \\
	\E X^{[p]}&=2\frac{(-1)^{\ell}}{\sin\pi\la}\,\Re(i^{p})\,\Im f^{(p)}(0).  \label{eq:sign-p}
\end{align}
\end{corollary} 

Concerning the factors $\Re(i^{p+1})$ and $\Re(i^p)$ in \eqref{eq:c.f.|X|^p,u=0}, \eqref{eq:c.f.X^[p],u=0}, \eqref{eq:abs-p}, and \eqref{eq:sign-p}, note that they can also be written as $-\sin\frac{\pi p}2$ and $\cos\frac{\pi p}2$, respectively.  


\subsection{On results by Brown and Laue} 


Formula \eqref{eq:0} \big(which is a special case of \eqref{eq:pin_pos}\big) was obtained, in somewhat different form, by Brown~\cite{brown:1970}, under the additional restriction that $p$ 
is not an odd integer. Also, as mentioned in \cite{pinelis:2011}, the constant-sign
argument used in the proof in \cite{brown:1970} does not actually seem to work for $p\in(0,2)$.
The mentioned representation was used in \cite{brown:1970} to prove, under a Lindeberg-type condition, the   
convergence of the absolute moments in the central limit theorem. 

Brown~\cite{brown:1972} 
also obtained formula \eqref{eq:c.f.|X|^p,u=0}, except that, in place of the Lebesgue integral $\int_0^\infty$ in \eqref{eq:c.f.|X|^p,u=0}, \cite{brown:1972} has the limit $\int_{0+}^{\infty-}$. 

Laue \cite[Theorem~2.2]{laue:1980} showed that, under the conditions of Corollary~\ref{th:short}, 
%
\begin{align}\label{exp:thm:2.2:laue}
	&\E |X|^{p}=
	\left\{
	\begin{alignedat}{2}
	& \frac1{\cos(\pi \la/2)} \Re\big((-1)^{\ell/2}\,f^{(p)}(0)\big)  &&\text{\ \ if $\ell$ is even}, \\ 
	&  \frac1{\sin(\pi \la/2)} \Re\big((-1)^{(\ell+1)/2}\,f^{(p)}(0)\big) &&\text{\ \ if $\ell$ is odd}.
	\end{alignedat}
	\right.
\end{align}
Note that the factors $(-1)^{\ell/2}$ for even $\ell$ and $(-1)^{(\ell+1)/2}$ for odd $\ell$ are real and hence can be moved out of the arguments of the function $\Re$ in \eqref{exp:thm:2.2:laue}. Then it will be easy to see that the right-hand side in \eqref{exp:thm:2.2:laue} is the same as that in \eqref{eq:abs-p}. 
The reference to \cite{laue:1980} was provided by M.\ Matsui \cite{muneya}. 


\subsection{On results by Zolotarev and von Bahr, and their applications}\label{zol}
Take any real $p>0$ that is not an even integer, and write it in the form $2m+q$, where $m:=\lfloor p/2\rfloor$, so that $q=p-2m\in(0,2)$. Assume that $\E|X|^p<\infty$. 

Replace all the instances of $u$, $r$, and $p$ in \eqref{eq:pin_pos} by $0$, $2m$, and $q$, respectively. At that, accordingly, $\ell=\ell(p)$ is replaced by $\tilde\ell:=\ell(q)$, which is $0$ or $1$ \big(depending on whether $q\in(0,1]$ or $q\in(1,2)$\big). Thus, one has 
\begin{equation}\label{eq:pin_pos-zol}
\E X_+^p=\frac{\Gamma (q+1)}{2 \pi\,(-1)^m} \int_0^{\infty }
\Big(\frac{(R_{\tilde\ell} f^{(2m)})(0;t)}{(it)^{q+1}}+\frac{(R_{\tilde\ell} f^{(2m)})(0;-t)}{(-it)^{q+1}}\Big) \d t.    
\end{equation}
In the latter identity, replace $X$ by $-X$ and, accordingly, $f$ by $\ov f$. 
Using now the identities $|X|^p=X_+^p+(-X)_+^p$, $f+\ov f=2\Re f$, and 
\begin{equation}\label{eq:R=R_0}
	\text{$(R_{\tilde\ell} \Re f^{(2m)})(0;\pm t)
=(R_0 \Re f^{(2m)})(0;\pm t)=\Re f^{(2m)}(t)-\Re f^{(2m)}(0)$\quad for\quad $\tilde\ell\in\{0,1\}$}
\end{equation}
and $t\in(0,\infty)$, one concludes that 
\begin{equation}\label{eq:pin_abs-zol}
\E|X|^p=2(-1)^{m+1}\frac{\Ga(q+1)}\pi\,\sin\frac{\pi q}2\, \int_0^{\infty }
\frac{\Re f^{(2m)}(t)-\Re f^{(2m)}(0)}{t^{q+1}}\, \d t;     
\end{equation}
the first equality in \eqref{eq:R=R_0} holds because, by (say) \eqref{eq:f^{(q)}=}, 
$\Re f^{(2m+1)}(0)=0$. 

Formula \eqref{eq:pin_abs-zol}, due to Zolotarev, can be found in 
\cite[page~394]{lukacs_russian}, where the factor $2(-1)^{m+1}$ is missing, though. 
A formula similar to \eqref{eq:pin_abs-zol} was implicit in \cite{zolot-mellin}, where it was used to find expressions of the absolute moments of stable laws. 

Similarly, using \eqref{eq:0} \big(which is the special case of \eqref{eq:pin_pos}, with $r=u=0$\big), one has the identity 
\begin{equation}\label{eq:pin_abs-bahr}
\E|X|^p=-\frac{2\Ga(p+1)}\pi\,\sin\frac{\pi p}2\, \int_0^{\infty }
\frac{\Re(R_\ell f)(0;t)}{t^{p+1}}\, \d t,      
\end{equation}
again for any real $p>0$ that is not an even integer, with $\ell=\ell(p)$. 
The latter identity was obtained by von Bahr \cite[Lemma~4]{bahr-converg65} in a slightly different form. In fact, von Bahr had the integral $\beta_{H;p}:=\int_{-\infty}^\infty|x|^p|\,\d H(x)$ and the Fourier--Stieltjes transform of $H$, respectively, in place of $\E|X|^p$ and $f$ in \eqref{eq:pin_abs-bahr}, where $H$ is any function of bounded variation on $(-\infty,\infty)$ with $\beta_{H;p}<\infty$. However, this formally more general result follows immediately from \eqref{eq:pin_abs-bahr}, since any function of bounded variation on $(-\infty,\infty)$ is a linear combination of two distribution functions. Identity \eqref{eq:pin_abs-bahr} was used in \cite{bahr-converg65} to provide bounds on the rate of convergence of the absolute moments of the standardized sums of independent r.v.'s to the corresponding moments of the standard normal distribution.

In the special case of $p\in(0,2)$, identity \eqref{eq:pin_abs-bahr} coincides with \eqref{eq:pin_abs-zol}. For $p\in(1,2)$, \eqref{eq:pin_abs-bahr} was used in \cite{bahr65} to obtain the famous von Bahr--Esseen (vBE) upper bound on the absolute moment of order $p$ of sums of independent zero-mean r.v.'s. More general and at that exact versions of the vBE bound were given recently in \cite{bahr-esseen-AFA_publ}.  

Identities \eqref{eq:pin_abs-zol} and \eqref{eq:pin_abs-bahr} were utilized in \cite{braverman86} and  \cite{bms95}, respectively, to characterize a probability distribution in terms of certain related absolute moments. 
For $p\in(0,2)$, identity \eqref{eq:pin_abs-zol}--\eqref{eq:pin_abs-bahr} was used in \cite{braverman93} to obtain upper and lower bounds on the absolute moments of order $p$ of sums of independent symmetric r.v.'s; those upper and lower bounds differ from each other by a factor depending only on $p$. In \cite{braverman96}, \eqref{eq:pin_abs-bahr} was applied to provide upper and lower bounds (again differing from each other by a factor depending only on $p$) on the Lorentz norm of a r.v.\ in terms of its c.f. 
Explicit values of the constant factors before the integrals in \eqref{eq:pin_abs-zol} and \eqref{eq:pin_abs-bahr} were not used in \cite{braverman86,bms95,braverman93,braverman96}. 

\subsection{Symmetric differences 
}\label{inv f'la}
Instead of derivatives of various orders, one can use symmetric
differences of the values of the c.f. 
In particular, this will show that Theorem~\ref{th:lin comb} implies certain extensions of results by Wolfe
\cite{wolfe:1973} and Gurland \cite{gurland:1948}. 

Take any $n\in\N$ and then any real $p\in\big(\frac{(-1)^n-1}2,n\big)$; the latter relation can be rewritten as  
\begin{equation}\label{eq:p,n}
	\text{$0<p<n$ if $n$ is even, and $-1<p<n$ if $n$ is odd. }
\end{equation} 
Consider the function $g$ defined by the formula 
\begin{equation}\label{eq:g_n}
g(t):=g_n(t):=(2i)^n\sin^n t=(e^{it}-e^{-it})^n
=(\De_{it}^n\exp)(0)  
\end{equation}
for all real $t$, where $\De_v^n$ is the $n$th power of the symmetric difference operator $\De_v$, defined by the formula $(\De_v\psi)(z):=\psi(z+v)-\psi(z-v)$, so that 
\begin{equation}\label{eq:De^n}
	(\De_v^n\psi)(z)=\sum_{j=0}^n(-1)^j\binom nj\psi\big(z+(n-2j)v\big) 
\end{equation}
for any function $\psi\colon\C\to\C$ and any $v$ and $z$ in $\C$.  

This function, $g=g_n$, satisfies the conditions \eqref{g=}, \eqref{p>}, and \eqref{O} of Theorem~\ref{th:lin comb}, with $q=n$; at that, all the $q_j$'s equal $0$. 
Expressions for the constants $c_p^\pm(g_n)$ are given in Appendix~\ref{c_p,n}, where is also shown that these constants are nonzero. Moreover, it is clear that the function $g_n$ is even or odd when 
$n$ is so, respectively; hence, 
\begin{equation}\label{eq:c_p^-(g_n)=}
	c_p^-(g_n)=(-1)^n c_p^+(g_n). 
\end{equation}
%
Thus, from \eqref{eq:lin comb} one immediately obtains 
%
\begin{corollary}\label{cor:g_n} 
Take any real $r$ and $u$, and any r.v.\ $X$ such that $\E|X|^r+\E|X|^{r+p}<\infty$.  
Let $f$ be the c.f.\ of $X$. 
Then 
\begin{equation}\label{eq:cor:g_n}
\E X_+^{r+p} e^{iuX}+(-1)^n e^{-i\pi r}\,\E X_-^{r+p} e^{iuX}
=	\frac1{i^r\,c^+_p(g_n)}\int_0^{\infty-}\frac{(\De_t^n f^{(r)})(u)}{t^{p+1}} \d t  
\end{equation}
\end{corollary}
%
%
and 

\begin{corollary}\label{cor:g_n,m} 
Take any $r\in\R$, $u\in\R$, and any r.v.\ $X$ such that $\E|X|^r+\E|X|^{r+p}<\infty$.  
Let $f$ be the c.f.\ of $X$.  
Take any $p\in(0,\infty)$. 
Finally, take any two natural numbers $n$ and $m$ in the interval $(p,\infty)$ such that $n$ is odd and $m$ is even. 
Then 
\begin{align}
\E X_+^{r+p} e^{iuX}
&=	\frac1{2i^r}\int_0^{\infty-}
\Big(\frac{(\De_t^m f^{(r)})(u)}{c^+_p(g_m)}+\frac{(\De_t^n f^{(r)})(u)}{c^+_p(g_n)}\Big) 
\frac{\d t}{t^{p+1}}, \label{eq:g_n,m,+,r=0} \\ 
\E X_-^{r+p} e^{iuX}
&=	\frac{e^{i\pi r}}{2i^r}\int_0^{\infty-}
\Big(\frac{(\De_t^m f^{(r)})(u)}{c^+_p(g_m)}-\frac{(\De_t^n f^{(r)})(u)}{c^+_p(g_n)}\Big) 
\frac{\d t}{t^{p+1}}. \label{eq:g_n,m,-,r=0}  
\end{align}
It follows, in particular, that 
\begin{align}
\E|X|^p e^{iuX}
&=\frac1{c^+_p(g_m)}\int_0^{\infty-}\frac{(\De_t^m f)(u)}{t^{p+1}} \d t, \label{eq:g_n,m,abs} \\ 
\E X^{[p]} e^{iuX}
&=\frac1{c^+_p(g_n)}\int_0^{\infty-}\frac{(\De_t^n f)(u)}{t^{p+1}} \d t. \label{eq:g_n,m,sign}  
\end{align}
In fact, by Corollary~\ref{cor:g_n}, the identity \eqref{eq:g_n,m,sign} holds even for any $p\in(-1,\infty)$ (and any odd $n>p$). 
\end{corollary}

 

In the special case $u=0$, formula
\eqref{eq:g_n,m,abs} was obtained in \cite[Theorem 2]{wolfe:1973}; however, the constant $c^+_p(g_m)$ was left unevaluated there. 


A special case of Corollary~\ref{cor:g_n}, with odd $n$ and $p=r=u=0$, is 

\begin{corollary}\label{cor:odd n,000} 
For any r.v.\ $X$, any real $x$, and any $m\in\{0,1,\dots\}$, 
\begin{equation}\label{eq:odd n,000}
\P(X\ltheq x)
=	\frac12-\frac{(-1)^m}{2\pi i}\int_0^{\infty-}(\De_t^{2m+1} f_x)(0)\, \frac{\d t}t\Big/\binom{2m}m,   
\end{equation}
where $f_x$ stands for the c.f.\ of the r.v.\ $X-x$, so that 
$f_x(t):=e^{-itx}f(t)$ for all real $x$ and $t$. 
\end{corollary}

To deduce Corollary~\ref{cor:odd n,000} from Corollary~\ref{cor:g_n}, note first that without loss of generality one may assume that here $x=0$. Use then the identity $\P(X>0)-\P(X<0)=1-2\P(X\ltheq0)$ and Corollary~\ref{cor:p=0}. 

Here and in the sequel, we employ the notation 
\begin{equation}\label{eq:half-eq}
\begin{alignedat}{2}
\P(X\ltheq b)&:=\P(X<b)+\tfrac12\P(X=b),\quad 
&	\P(X\gtheq b)&:=\P(X>b)+\tfrac12\P(X=b),
\\
	\ii\{a\ltheq b\}&:=\ii\{a<b\}+\tfrac12\ii\{a=b\},\quad 
&	\ii\{a\gtheq b\}&:=\ii\{a>b\}+\tfrac12\ii\{a=b\}
\end{alignedat}
\end{equation}
for real $a$ and $b$, 
so that one can read the symbol $\ltheq$ as ``less than or half-equal to'', and similarly for $\gtheq$. Such a symmetric half-and-half splitting of an atom at a given point is natural and convenient in Fourier analysis, since 
\begin{equation}\label{eq:sin int}
	\frac1{2\pi}\int_
	{-\infty+}^{\infty-}\frac{\sin tx}t \d t=\tfrac12\,\sgn x=\ii\{x\gtheq0\}-\tfrac12=\tfrac12-\ii\{x\ltheq0\}
\end{equation}
for all $x\in\R$
; cf.\ e.g.\ \cite{pitman61}. 
Of course, if $b\in\R$ is not an atom of the distribution of $X$, then $\P(X\ltheq b)=\P(X<b)=\P(X\le b)$ and $\P(X\gtheq b)=\P(X>b)=\P(X\ge b)$. 

Thus, \eqref{eq:odd n,000} may be considered as an inversion formula for the distribution function, say $F$, of the r.v.\ $X$ -- regularized so that $F(x)=\frac12\,\big(F(x+)+F(x-)\big)
=\P(X\ltheq x)$ for all $x\in\R$. 

A special case of the inversion formula \eqref{eq:odd n,000}, with $m=0$, is the following result due to Gurland \cite[(2)]{gurland:1948}:  
For any r.v.\ $X$ and any real $x$,  
\begin{equation}\label{eq:gurland48}
\P(X\ltheq x)-\frac12
=-\frac1{\pi}\int_0^{\infty-}\Im[e^{-itx}f(t)]\, \frac{\d t}t 
=-\frac1{2\pi i}\,\pv\int_{-\infty}^\infty e^{-itx}f(t)\, \frac{\d t}t,     
\end{equation}
where $\pv$ stands for ``principal value'', so that $\pv\int_{-\infty}^\infty
:=\lim_{\vep\downarrow0,\;T\uparrow\infty}
\big(\int_{-T}^{-\vep}+\int_\vep^T\big)$. 


On the other hand, one may note the following generalization of Corollary~\ref{cor:odd n,000}, which provides expressions of truncated moments of any real order $r$ in terms of the c.f. 

\begin{corollary}\label{cor:odd n,0r0} 
For any $m\in\{0,1,\dots\}$, any real $r$ and $x$, and any r.v.\ $X$ such that $\E|X|^r<\infty$, 
\begin{align*}
\E\ov{X^r}\ii\{X\ltheq x\}
=&	\frac{f^{(r)}(0)}{2i^r}-\frac{(-1)^m}{2\pi i^{r+1}}\int_{0+}^{\infty-}\big(\De_t^{2m+1}f^{(r)}_x\big)(0)\, \frac{\d t}t\Big/\binom{2m}m,  \\  
\E\ov{X^r}\ii\{X\gtheq x\}
=&	\frac{f^{(r)}(0)}{2i^r}+\frac{(-1)^m}{2\pi i^{r+1}}\int_{0+}^{\infty-}\big(\De_t^{2m+1}f^{(r)}_x\big)(0)\, \frac{\d t}t\Big/\binom{2m}m,    
\end{align*}
where $f^{(r)}_x(t):=e^{-itx}f^{(r)}(t)$ for all real $x$ and $t$. 
\end{corollary}

To obtain this corollary, take \eqref{eq:E YX^{[p]}} with $p=0$, replace there $g_1$ by $g_{2m+1}$ and $X$ by $X-x$, then replace $Y$ by $\ov{X^r}$, and then use \eqref{eq:f^{(q)}=}, Corollary~\ref{cor:p=0}, and the identities $(X-x)^{[0]}=\ii\{X\gtheq x\}-\ii\{X\ltheq x\}=2\ii\{X\gtheq x\}-1=1-2\ii\{X\ltheq x\}$. 
One may also note here that, similarly to Corollary~\ref{cor:odd n,000}, each of the two instances of the integral $\int_{0+}^{\infty-}$ in Corollary~\ref{cor:odd n,0r0} may be replaced by $\int_0^{\infty-}$. 

\begin{remark}
In place of the function $g_n$ as in \eqref{eq:g_n}, one can similarly consider its more general version given by the formula $g(t)=g_{n;a,b}(t):=(e^{ibt}-e^{iat})^n$ for arbitrary distinct real $a$ and $b$ and all real $t$. 
\end{remark}

\section{Computational aspects}
\label{comput}

There may be two problems with the calculation of the integral in \eqref{eq:pin_pos} (and hence with the integrals in subsequent corollaries): 
\begin{enumerate}[Problem~1:]
	\item For small values of $t$, the terms $f^{(r)}(u\pm t)$ and $-\sum_{j=0}^\ell f^{(r+j)}(u)\,\frac{(\pm t)^j}{j!}$, 
comprising $\big(R_\ell f^{(r)}\big)(u;\pm t)$ in accordance with \eqref{eq:R_m}, may nearly cancel each other. 
%
Especially when $p$ and hence $\ell$ are large, this near-cancellation effect, together with the smallness of the denominator $t^{p+1}$ of the integrand, may cause instability. 
	\item The integral in \eqref{eq:pin_pos} converges rather slowly near $\infty$, since the integrand is asymptotically proportional (as $t\to\infty$) to $1/t^{p+1-\ell}$ and the exponent $p+1-\ell$ is no greater than $2$ and may be arbitrarily close to $1$ if $\ell$ is close to $p$. 
\end{enumerate}

These two computational problems can be rather easily resolved by taking some $b\in(0,\infty)$ and writing the integral $\int_0^\infty$ in \eqref{eq:pin_pos} as $\int_0^b+\int_b^\infty$, of course with the same integrand; the choice $b=1$ appears to work well in most cases. 
The integral $\int_0^b$ can then be evaluated by a repeated integration by parts (or, more conveniently, by the Fubini theorem), according to the formula 
%
\begin{equation}\label{eq:recur}
	\int_0^b\frac{\big(\tR_{p,m}f^{(\ell-m+r)}\big)(u;t)}{t^{p+1-\ell+m}}\,\d t
	=-\frac{\big(\tR_{p,m}f^{(\ell-m+r)}\big)(u;b)}{(p-\ell+m)b^{p-\ell+m}}
	+\int_0^b\frac{\big(\tR_{p,m-1}f^{(\ell-m+r+1)}\big)(u;t)}{(p-\ell+m)t^{p-\ell+m}}\,\d t,
\end{equation}
%
%
where $m=\ell,\ell-1,\dots$ and 
$$
\big(\tR_{p,j}\psi\big)(u;t):=
\frac{\big(R_j\psi\big)(u;t)}{i^{p+1}}
   +(-1)^{\ell-j}\frac{\big(R_j\psi\big)(u;-t)}{(-i)^{p+1}},$$ 
provided that $p-\ell+m\ne0$ and 
\begin{equation}\label{eq:Fub}
\E|X|^{\ell-m+r}+\E|X|^
{p+r}+\E|X|^{\ell-m+r+1}<\infty; 	
\end{equation}
the justification of 
this application of the Fubini theorem is done similarly to that in 
the reasoning preceding Corollary~\ref{cor:pin_pos}; note here that 
for $p\ge1$ \eqref{eq:Fub} and hence \eqref{eq:recur} hold for all $m\in\{1,\dots,\ell\}$, by the condition $\E|X|^r+\E|X|^{r+p}<\infty$ and the definition of $\ell$ in \eqref{eq:ell}. 

Each application of the identity \eqref{eq:recur} reduces both (i) the degree of the near-cancellation effect in the numerator of the integrand and (ii) the degree of the smallness of the denominator of the integrand. 

As for the $\int_b^\infty$ ``part'' of the integral $\int_0^\infty$ in
\eqref{eq:pin_pos}, the terms of the integrand causing the slow
convergence near $\infty$ are power terms, proportional to
$\frac{t^j}{t^{p+1}}$ for $j\in\{0,\dots,\ell\}$, whose integrals are
most easily evaluated, since $\int_b^\infty \frac{t^j}{t^{p+1}}\,\d
t=\frac1{(p-j)b^{p-j}}$. 

Thus, 
under the conditions of Corollary~\ref{cor:pin_pos} 
with $p\ge1$,
\begin{equation}\label{eq:any m}
	\begin{aligned}
& \int_0^{\infty }
\Big(\frac{(R_\ell f^{(r)})(u;t)}{(it)^{p+1}}+\frac{(R_\ell f^{(r)})(u;-t)}{(-it)^{p+1}}\Big) \d t \\ 
&\!\!\begin{multlined}[t][.9\displaywidth]
	=\frac{\Ga (p-\ell+m)}{\Ga (p+1)} 
	\int_0^b\frac{\big(\tR_{p,m-1}f^{(\ell-m+r+1)}\big)(u;t)}{t^{p-\ell+m}}\,\d t
\\
    -\frac1{\Ga (p+1)}\sum _{j=m}^{\ell} 
   \frac{\Ga(p-\ell+j)}{b^{p-\ell+j}} 
   \big(\tR_{p,j}f^{(\ell-j+r)}\big)(u;b)\ \ 
\end{multlined}   
\\
&\ \ \ +\int_b^{\infty }
   \Big(\frac{f^{(r)}(
   u+t)}{i^{p+1}}+\frac{f^{(r)}(u-t)}{(-i)^{p+1}}\Big) \frac{\d t}{t^{p+1}}
   -\sum _{j=0}^{\ell}
   \Big(\frac{1}{i^{p+1}}+\frac{(-1)^j}{(-i)^{p+1}}\Big)
   \frac{f^{(j+r)}(u)}{j! (p-j)b^{p-j}} 	
\end{aligned}
\end{equation}
for any $b\in(0,\infty)$ and any $m\in\{1,\dots,\ell
\}$. 
In fact, representation \eqref{eq:any m} holds as well 
for any 
$p\in(0,\infty)\setminus\N$ and any $m\in\{0,-1,\dots\}$ such that $\E|X|^{\ell-m+r+1}<\infty$; 
note here that for $m\in\{-1,-2,\dots\}$ the factor $\frac1{t^{p-\ell+m}}$ in the first integral in \eqref{eq:any m} is no longer singular in $t$ \big(near $0$ or anywhere on the interval $(0,b)$\big). 
Representation \eqref{eq:any m} also holds for $m=\ell+1$, in which case it is tautological as far as the $\int_0^b$ part of the integral on the left-hand side of \eqref{eq:any m} is concerned.  

As usual, we assume here and in the sequel that the sum of any empty family is $0$, so that 
$\sum _{j=\ell+1}^{\ell}a_j=0$ and 
$\sum_{j=0}^{m-1} a_j=0$ for $m\in\{0,-1,\dots\}$ and any $a_j$'s in $\C$. 
Thus, in accordance with \eqref{eq:R_m}, the terms 
$\big(\tR_{p,m-1}f^{(\ell-m+r+1)}\big)(u;t)$ in the integrand of the integral from $0$ to $b$ in \eqref{eq:any m} can be written simply as 
$\big(\frac1{i^{p+1}}
   +\frac{(-1)^{\ell-m+1}}{(-i)^{p+1}}\big)f^{(\ell-m+r+1)}(u)$ 
for any $m\in\{0,-1,\dots\}$. 

Whereas the right-hand 
side of \eqref{eq:any m} looks much more complicated than its left-hand side, the representation \eqref{eq:any m} provides for fast and instability-free calculations. 

%




\appendix

\section{The constants \texorpdfstring{$c_p^\pm(g_n)$}{}}\label{c_p,n}
Here we shall present explicit expressions of the constants $c_p^\pm(g_n)$ for the function $g_n$ as in \eqref{eq:g_n}; at this point, recall the definitions \eqref{eq:c^pm} of $c_p^\pm(g)$. 

  

For any 
$p\in\R\setminus\Z$, let 
\begin{equation}\label{eq:ka_p}
\ka_p:=\Ga(-p)=-\frac\pi{\Ga(p+1)\sin\pi p}. 
\end{equation}


\begin{lemma}\label{lem:int e_ell}
For any $x\in\R$, any $p\in(0,\infty)\setminus\N$, 
and $\ell$ and $e_\ell$ as in \eqref{eq:ell}, 
\begin{align}\label{eq:cont:integ:a21}
\int_0^\infty \frac{e_\ell(itx)}{t^{p+1}}\d t=\ka_p\,i^p\,\ov{(-x)^p}=\ka_p\,(i\sgn x)^{-p}\,|x|^p;  
\end{align}
the latter expression is assumed to equal $0$ when $x=0$. 
Moreover, for any $x\neq 0$ and any $p\in (-1,0)$,
\begin{align}\label{eq:cont:integ:a22}
\int_0^
{\infty-} \frac{e^{itx}}{t^{p+1}}\d t =\ka_p\,i^p\,\ov{(-x)^p}=\kappa_p(i\sgn x)^{-p}|x|^p. 
\end{align} 
\end{lemma}

\begin{proof} 
First, note that identity \eqref{eq:cont:integ:a22} for $p\in(-1,0)$ follows immediately from \eqref{eq:cont:integ:a21} for $p\in(0,1)$ via integration by parts. 

So, it remains to verify \eqref{eq:cont:integ:a21} for $p\in(0,\infty)\setminus\N$. 
Let us denote the integral in \eqref{eq:cont:integ:a21} by $I(x)$. For $x=0$, identity \eqref{eq:cont:integ:a21} is trivial. So, by the positive homogeneity in $x$, without loss of generality one may assume that $x\in\{1,-1\}$. Moreover, clearly $I(-1)=\ov{I(1)}$. 

Combining \eqref{eq:e_ell=R}, \eqref{eq:0} for $X=1$ and $X=-1$, \eqref{eq:c_pm,pin11}, \eqref{eq:g pin_pos}, and \eqref{eq:c^pm}, one obtains the system of equations 
\begin{align*}
	i^{-p-1}I(1)+(-i)^{-p-1}I(-1)=&\,2\pi/\Ga(p+1), \\ 
	(-i)^{-p-1}I(1)+i^{-p-1}I(-1)=&\,0 
\end{align*}
for $I(1)$ and $I(-1)$, whence \eqref{eq:cont:integ:a21} for $x\in\{1,-1\}$ follows.  
\end{proof}


\begin{lemma}\label{lem:cont}
For each $n\in\N$, the constants $c_p^\pm(g_n)$ depend continuously on  $p\in\big(\frac{(-1)^n-1}2,n\big)$; here one may recall \eqref{eq:p,n}. 
\end{lemma}

\begin{proof}
Take any real $p_1$ and $p_2$ such that $\frac{(-1)^n-1}2<p_1<p_2<n$. 

Note that $\big|\frac{g_n(t)}{t^{p+1}}\big|=2^n\big|\frac{\sin^n t}{t^{p+1}}\big|=O(t^{n-p_2-1})$ over all $p\in[p_1,p_2]$ and $t\in(0,1]$. 
Also, $\int_0^1 t^{n-p_2-1}\d t<\infty$, since $p_2<n$. 
So, by dominated convergence, 
\begin{equation}\label{eq:int_0^1}
\text{$\int_0^1 \frac{g_n(t)}{t^{p+1}}\d t$ is continuous in $p\in[p_1,p_2]$.}
\end{equation}

Concerning the ``remaining'' part, $\int_1^{\infty-}\frac{g_n(t)}{t^{p+1}}\d t$, of the integral 
\begin{equation}\label{eq:c_p(g_n)}
c_p^+(g_n)=\int_0^{\infty-}\frac{g_n(t)}{t^{p+1}}\d t, 	
\end{equation}
consider separately the two cases depending on whether $n$ is even or odd. 
If $n$ is even then $p_1>0$, and so, $\int_1^\infty t^{-p_1-1}\d t<\infty$, whereas $\big|\frac{g_n(t)}{t^{p+1}}\big|\le2^n t^{-p_1-1}$ for all $p\in[p_1,p_2]$ and $t\in[1,\infty)$. 
So, by dominated convergence, 
\begin{equation}\label{eq:int_1^infty}
\text{$\int_1^{\infty-} \frac{g_n(t)}{t^{p+1}}\d t$ is continuous in $p\in[p_1,p_2]$}
\end{equation}
-- if $n$ is even \big(in which case the limit $\int_1^{\infty-}$ in \eqref{eq:int_1^infty} could actually be written as the Lebesgue integral $\int_1^\infty$\big). 

Finally, if $n$ is odd then, by \eqref{eq:g_n} and \eqref{eq:De^n}, $g_n=G_n'$, where 
$G_n(t):=\sum_{j=0}^n(-1)^j\binom nj \frac{e^{i(n-2j)t}}{i(n-2j)}=O(1)$ over all $t\in\R$. 
So, integrating by parts, one has 
\begin{equation*}
	\int_1^{\infty-} \frac{g_n(t)}{t^{p+1}}\d t
	=G_n(1)+(p+1)\int_1^{\infty-} \frac{G_n(t)}{t^{p+2}}\d t. 
\end{equation*}
Next, $\big|\frac{G_n(t)}{t^{p+2}}\big|=O(t^{-p_1-2})$ over all $p\in[p_1,p_2]$ and $t\in[1,\infty)$. 
Also, 
$\int_1^\infty t^{-p_1-2}\d t<\infty$, since $p_1>-1$. 
So, by dominated convergence, \eqref{eq:int_1^infty} holds for odd $n$ as well. 

Thus, 
in view of \eqref{eq:c_p(g_n)}, \eqref{eq:int_0^1},  and \eqref{eq:int_1^infty}
, 
$c_p^+(g_n)$ is continuous in  $p\in\big(\frac{(-1)^n-1}2,n\big)$. 
To complete the proof of Lemma~\ref{lem:cont}, it remains to recall \eqref{eq:c_p^-(g_n)=}. 
\end{proof}

%

\begin{proposition}\label{prop:c_p(g_n)}
Take any $n\in\N$ and then any real $p\in\big(\frac{(-1)^n-1}2,n\big)$ such that $n-p$ is \emph{not} an even integer. Then 
\begin{equation}\label{eq:c_p(g_n)=}
	c^+_p(g_n)=-\frac\pi{\Ga(p+1)}\,
	\frac{\si_{n,p}}{\rho_{n,p}}, 
\end{equation}
where 
\begin{equation*}
		\rho_{n,p}:=
		\left\{
		\begin{alignedat}{2}
		&\sin(\pi p/2)&&\text{ if $n$ is even}, \\
		&i\cos(\pi p/2)\;&&\text{ if $n$ is odd} 
		\end{alignedat}
		\right.
\end{equation*}
and 
\begin{equation}\label{eq:si}
		\si_{n,p}:=\sum_{0\le j<n/2}(-1)^j\binom nj\,(n-2j)^p. 
\end{equation}
\end{proposition}

\begin{proof}
Consider first the case when $p\notin\N$. 
Recall \eqref{eq:De^n} and note that 
\begin{equation}\label{eq:poly,0}
	\text{$\De_v^n\psi=0$ for any polynomial $\psi$ of degree less than $n$.} 
\end{equation}
So, by \eqref{eq:g_n}, \eqref{eq:ell}, and \eqref{eq:p,n}, 
\begin{equation*}
	g_n(z)=(\De_{iz}^n e_\ell)(0)
\end{equation*}
for all $z\in\C$. 
Therefore, by \eqref{eq:c^pm}, \eqref{eq:De^n}, and Lemma~\ref{lem:int e_ell},  
\begin{align*}
	c^+_p(g_n)
	&=\ka_p\sum_{j=0}^n(-1)^j\binom nj\,\big(i\sgn(n-2j)\big)^{-p}\,|n-2j|^p \\ 
	&=\ka_p\left(i^{-p}\,\sum_{0\le j<n/2}(-1)^j\binom nj\,(n-2j)^p
	+(-i)^{-p}\,\sum_{n/2<m\le n}(-1)^m\binom nm\,(2m-n)^p\right) 
	\\ 
&=\ka_p\big(i^{-p}+(-1)^n(-i)^{-p}\big)\si_{n,p}. 
\end{align*}
So, in view of 
\eqref{eq:ka_p}, \eqref{eq:c_p(g_n)=} follows -- in the case when $p\notin\N$. 

To finish the proof, note that $\rho_{n,p}\ne0$ if $n-p$ is not an even integer. 
Thus, the general case of 
Proposition~\ref{prop:c_p(g_n)}, with any real $p\in\big(\frac{(-1)^n-1}2,n\big)$ such that $n-p$ is not an even integer, now follows by Lemma~\ref{lem:cont}.  
\end{proof}

Proposition~\ref{prop:c_p(g_n)} is complemented by 

\begin{proposition}\label{prop:c_p(g_n),compl}
Take any $n\in\N$ and then any real $p\in\big
(\frac{(-1)^n-1}2,n\big)$ 
such that $n-p$ \emph{is} an even integer. Then ($p\in
\N$ and)
\begin{equation}\label{eq:c_p(g_n)=,compl}
	c^+_p(g_n)=
	(-1)^{1+\lfloor p/2\rfloor}\,\frac2{p!}\,
	\frac{\si'_{n,p}}{\ga_n}, 
\end{equation}
where 
\begin{equation*}
		\ga_n:=
		\left\{
		\begin{alignedat}{2}
		&1&&\text{ if $n$ is even}, \\
		-&i\;&&\text{ if $n$ is odd} 
		\end{alignedat}
		\right.
\end{equation*}
and 
\begin{equation*}
		\si'_{n,p}:=\frac{\partial\si_{n,p}}{\partial p}=\sum_{0\le j<n/2}(-1)^j\binom nj\,(n-2j)^p\,\ln(n-2j). 
\end{equation*}
\end{proposition}

\begin{proof}
It is clear that $p\in
\N$, since $n\in\N$ and $n-p$ is an even integer. 
Moreover, it follows that $\rho_{n,p}=0$ and also $\si_{n,p}=\frac12\,\sum_{0\le j\le n}(-1)^j\binom nj\,(n-2j)^p=0$, by \eqref{eq:si} and \eqref{eq:poly,0}. 
So, Proposition~\ref{prop:c_p(g_n),compl} follows from Proposition~\ref{prop:c_p(g_n)} by Lemma~\ref{lem:cont} and l'Hospital's rule, since 
$\frac{\partial\rho_{n,p}}{\partial p}=(-1)^{\lfloor
 p/2\rfloor}\,\frac\pi2\,\ga_n$ at any point $p\in
 \N$. 
\end{proof}

\begin{remark}\label{rem:parity}
The respective explicit expressions for $c^-_p(g_n)$ immediately follow from Propositions~\ref{prop:c_p(g_n)} and \ref{prop:c_p(g_n),compl} and the relation \eqref{eq:c_p^-(g_n)=}. 
\end{remark}


A special case of Proposition~\ref{prop:c_p(g_n)}, with odd $n$ and $p=0$, is given in 

\begin{corollary}\label{cor:p=0}
For any $m\in\{0,1,\dots\}$ 
\begin{equation}\label{eq:p=0}
	c^+_0(g_{2m+1})=(-1)^m\,\binom{2m}m\,i\pi=-c^-_0(g_{2m+1}). 
\end{equation}
\end{corollary}

\begin{proof}
Using the identity $\binom{2m+1}j=\binom{2m}j+\binom{2m}{j-1}$, write 
\begin{align*}
	\si_{2m+1,0}&=\sum_{j=0}^m (-1)^j\binom{2m+1}j \\ 
	&=\sum_{j=0}^m (-1)^j\,\Big[\binom{2m}j+\binom{2m}{j-1}\Big] \\ 
	&=\sum_{j=0}^m (-1)^j\binom{2m}j-\sum_{r=-1}^{m-1} (-1)^
	r\binom{2m}r
	=(-1)^m\,\binom{2m}m.
\end{align*}
Now the first equality in \eqref{eq:p=0} follows by Proposition~\ref{prop:c_p(g_n)}, and the second equality there holds in view of Remark~\ref{rem:parity}.  
\end{proof} 

The equality \eqref{eq:p=0} can be rewritten as 
$
	\int_0^{\infty-}\frac{\sin^{2m+1} t}t\d t=\frac\pi{2^{2m+1}}\,\binom{2m}m, 
$ 
again for any $m\in\{0,1,\dots\}$. 
A special case of this, for $m=0$, is the well-known equality 
$
	\int_0^{\infty-}\frac{\sin t}t\d t=\frac\pi2. 
$ 
The special case of Proposition~\ref{prop:c_p(g_n)} with $n=1$ (and hence $p\in(-1,1)$), which can be written as 
$
 \int_0^{\infty-} \frac{\sin t}{t^{p+1}}\d t =
 \frac{\pi/2}{\Gamma(p+1)\cos (\pi p/2)},  
$ 
is also known; see e.g.\ 
\cite[page~428]{pitman:1968}. 




We conclude this appendix by the following proposition, which shows that the constants $c_p^\pm(g_n)$ are nonzero. 
\begin{proposition}\label{prop:c ne0}
For all $n\in\N$ and $p$ as in \eqref{eq:p,n}, one has 
$
	c_{p,n}:=(\pm2i)^n c_p^\pm(g_n)>0. 
$ 
\end{proposition}

\begin{proof}
By \eqref{eq:g_n} and \eqref{eq:c_p^-(g_n)=}, 
\begin{equation*}
	c_{p,n}=\int_0^{\infty-}\frac{\sin^n t}{t^{p+1}}\,\d t. 
\end{equation*}
So, obviously $c_{p,n}>0$ if $n$ is even. 
If $n$ is odd, then $c_{p,n}=(I_0-I_1)+(I_2-I_3)+\dots$, where 
\begin{equation*}
	I_j:=(-1)^j\int_{j\pi}^{(j+1)\pi}\frac{\sin^n t}{t^{p+1}}\,\d t
	=\int_0^\pi\frac{|\sin u|^n}{(j\pi+u)^{p+1}}\,\d u, 
\end{equation*}
which is strictly decreasing in $j\ge0$ to $0$. 
So, $c_{p,n}>0$ for odd $n$ as well.  
\end{proof}

\section{On the fractional derivatives}\label{sec:appd:B}
Our definitions of the (fractional) derivative of any real order (which go back to \cite{marchaud27}) follow 
 \cite[Eq.~(2.1)]{laue:1980} 
 for $p\in(0,\infty)\setminus \N$ and
\cite{wolfe:1975a} for other cases. 
Let $g:\R\to\C$. 
Define the derivative of $g$ of order $p$ at a real point $t$ as follows. 
\begin{enumerate}[(i)]
\item For $p\in\{0\}\cup\N$, the derivative is defined as usual: $g^{(p)}(t):=\frac{\d^p g(t)}{\d t^p}$, with $g^{(0)}(t):=g(t)$. 
	\item For $p\in(0,\infty)\setminus \N$ with $k=\lfloor p \rfloor$ and
$\lambda:=p-k$, 
\begin{align}\label{eq:D^p}
g^{(p)}(t):=
-\frac1{\Ga(-\la)}
\int_{-\infty}^t \frac{g^{(k)}(t)-g^{(k)}(u)}{(t-u)^{1+\la}}\,\d u   
=\frac1{\Ga(-\la)}
\int_0^\infty \frac{g^{(k)}(t-s)-g^{(k)}(t)}{s^{1+\la}}\,\d s.  
\end{align}
\item for $p\in(-1,0)$,
\begin{align}\label{eq:D^p,-1<p<0}
 g^{(p)}(t):= \frac{1}{\Gamma(-p)}\int_{-\infty
 +}^t
 \frac{g(u)}{(t-u)^{1+p}}\d u
 =\frac1{\Ga(-p)}
\int_0^{\infty-} \frac{g(t-s)}{s^{1+p}}\,\d s.
\end{align}
\item For $p\in 
\Z\setminus 
\{0\}\setminus\N$, 
\begin{equation}\label{eq:p in Z_-}
g^{(p)}(t):=\lim_{q\downarrow p}g^{(q)}(t).  	
\end{equation}
\item For $p\in(-\infty,0)\setminus \Z$ 
\begin{equation}\label{eq:p notin Z_-}
g^{(p)}(t):=(g^{(\lceil p\rceil)})^{(p-\lceil p\rceil)}(t).  	
\end{equation}
\end{enumerate}
Of course, for the so defined fractional derivatives to exist, all the operations that need to be done in order to evaluate the corresponding expression have to be applicable. 

It will be seen that, for $g^{(p)}$ to exist for a given real $p$, it is enough to assume that 
\begin{equation}\label{eq:g=f}
	g=f,\quad \text{where $f$ is the c.f.\ of a r.v.\ $X$ with $\E|X|^p<\infty$.}
\end{equation}

Note
that part (i) of the above definition of the fractional derivative 
is different from
that of \cite{wolfe:1975a}, 
which is 
$
 \frac{d^k g^{(p-k)}(t)}{dt^k}
$
with $k= \lfloor p \rfloor$. However, under condition \eqref{eq:g=f},  
these two definitions are equivalent. 
This follows immediately from 


\begin{lemma}
\label{lem:ap:b}
Let $p$, $k$, and $\la$ be as \eqref{eq:D^p}, and let 
$f$ and $X$ be as in \eqref{eq:g=f}. Then 
\begin{equation}\label{eq:diff}
	\frac{\d^k}{\d t^k}\int_0^\infty \frac{f(t-s)-f(t)}{s^{1+\la}}\,\d s
	=\int_0^\infty \frac{f^{(k)}(t-s)-f^{(k)}(t)}{s^{1+\la}}\,\d s
\end{equation}
for all real $t$.  
\end{lemma}

\begin{proof}
By the known rule of differentiation under the integral sign (see e.g.\ Theorem~(2.27)(b) in  
\cite{folland}), it is enough to show that the modulus of the integrand on the right-hand side of \eqref{eq:diff} is bounded uniformly in $t\in\R$ by an integrable function. But this follows because 
$$|f^{(k)}(t-s)-f^{(k)}(t)|\le\E|X|^k\,|e^{is|X|}-1|
$$ 
for all $t\in\R$ and $s\in(0,\infty)$, and 
$$\int_0^\infty\E|X|^k\,|e^{is|X|}-1|\,
\frac{\d s}{s^{1+\la}}=\E|X|^{k+\la}\int_0^\infty|e^{iv}-1|\,
\frac{\d v}{v^{1+\la}}<\infty.$$ 
\end{proof}


We 
shall now prove identity \eqref{eq:f^{(q)}=}, assuming accordingly that 
$f$ is the c.f.\ of a r.v.\ $X$ with $\E|X|^p<\infty$.  
As mentioned before, this identity was
presented in \cite{wolfe:1975a} (as Theorem~2 therein); note that notation ${}_{-\infty}\D_t^p f$ was used in
\cite{wolfe:1975a} in place of $f^{(p)}(t)$. 
The proof of \eqref{eq:f^{(q)}=} in \cite{wolfe:1975a} was given only for $p\in(-1,0)$, and even in that case there was a gap. Therefore, we shall provide here the necessary details. 

\begin{enumerate}[(I)]
	\item First, one has to justify the 
	interchangeability of the integration and expectation in \cite[page~311, lines 2--3]{wolfe:1975a} -- that is, the equality, for $p\in(-1,0)$, under the question mark in  
\begin{align}\label{eq:?}
 \Gamma(-p)\,f^{(p)}(u) 
 = \int_{0+}^{\infty-} \frac{\E e^{i(u-t)X}}{t^{p+1}}\d t
\overset{\text{?}}= \E \int_{0+}^{\infty-} \frac{e^{i(u-t)X}}{t^{p+1}}\d t  
\end{align}
for all real $u$. However, this can be done by dominated convergence, just as in the first line of \eqref{eq:lim} in the proof of Theorem~\ref{th:homo} -- in this case, for $Y=e^{iuX}$ and the function $g\in\G_p$ given by the formula $g(t)=e^{-it}$ for real $t$. 
Also, it is clear that one can replace here $\int_{0+}^{\infty-}$ by $\int_0^{\infty-}$. 

\item The case of $p=-1$ follows immediately from that of $p\in(-1,0)$ by dominated convergence, since 
$|\ov{X^p}e^{itX}| \le |X|^{-1} \vee1$ for all $p\in(-1,0)$. 

\item To treat the case $p<-1$, proceed by induction. Assume that \eqref{eq:f^{(q)}=} holds for some negative integer $-n$ in place of $p$. Take any $p\in(-n-1,-n)$, so that $-n=\lceil p\rceil$ and $p+n\in(-1,0)$. Then 
\begin{align*}
 f^{(p)}(t) &= (f^{-n})^{(p+n)}(t) \\
 &= \frac1{\Ga(-p-n)} \int_0^{\infty-} \,\frac{i^{-n}\E\ov{X^{-n}}e^{i(t-s)X}}{s^{1+p+n}} \d s \\ 
 &= \frac1{\Ga(-p-n)}\, i^{-n}\E\ov{X^{-n}}\,e^{itX}\int_0^{\infty-} \frac{e^{-isX}}{s^{1+p+n}} \d s \\ 
 &= i^p\E\ov{X^p}e^{itX}.   
\end{align*}
The first equality here holds by the definition in \eqref{eq:p notin Z_-}, the second one by the definition in  \eqref{eq:D^p,-1<p<0} and the assumed identity in \eqref{eq:f^{(q)}=} for $-n$ in place of $p$, the third one is verified quite similarly to \eqref{eq:?}, 
and the last equality holds by \eqref{eq:cont:integ:a22}. 
Hence, \eqref{eq:f^{(q)}=} holds for $p\in(-n-1,-n)$. 

It now follows by dominated convergence that \eqref{eq:f^{(q)}=} holds for $p=-n-1$; cf.\ the above consideration of the case $p=-1$ in item (II). 

Thus, by induction, \eqref{eq:f^{(q)}=} has been verified for all real $p<-1$ and, in view of items (I) and (II), for all real $p<0$. 

\item It remains to consider the case $p\in[0,\infty)$. The case $p\in\{0\}\cup\N$ is a textbook one. So, suppose that $p\in(0,\infty)\setminus \N$. 
Then, with $k$ and $\la$ as in \eqref{eq:D^p}, 
\begin{align*}
 f^{(p)}(t) 
 &= \frac1{\Ga(-\la)}\int_0^\infty \frac{f^{(k)}(t-s)-f^{(k)}(t)}{s^{1+\la}}\,\d s \\ 
&= \frac1{\Ga(-\la)}\int_0^\infty \frac{i^k \E X^k e^{itX}(e^{-isX}-1)}{s^{1+\la}}\,\d s \\ 
 &= \frac1{\Ga(-\la)}\, i^k\E\ov{X^k}\,e^{itX}\int_0^\infty \frac{e^{-isX}-1}{s^{1+\la}} \d s \\ 
 &= i^p\E\ov{X^p}e^{itX}.   
\end{align*}
The second equality here holds by the mentioned identity in \eqref{eq:f^{(q)}=} for the nonnegative integer $k$ in place of $p$, the third one is verified again quite similarly to \eqref{eq:?}, 
and the last equality holds by \eqref{eq:cont:integ:a21} (with $0$ and $\la$ in place of $\ell$ and $p$, respectively). 
Hence, \eqref{eq:f^{(q)}=} holds for all $p\in[0,\infty)$ as well, and thus indeed for all real $p$. 
\end{enumerate}

\bibliographystyle{abbrv}
\bibliography{C:/Users/ipinelis/Dropbox/mtu/bib_files/citations12.13.12}

\end{document}